\documentclass[reqno]{amsart}
\usepackage{amsmath,amssymb, amsthm}
\usepackage{hyperref}
 \newtheorem{theorem}{Theorem}[section]
\newtheorem{proposition}{Proposition}[section]

\newtheorem{lemma}{Lemma}[section]
 \newtheorem{remark}{Remark}[section]
\usepackage{hyperref}

\usepackage{color}

\begin{document}

\def\eps{\varepsilon}

\title[Hardy-H\'enon parabolic equations]
{Singularity and blow-up estimates via Liouville-type theorems for  Hardy-H\'enon parabolic equations}

\author{Quoc Hung PHAN}

\address{Universit\'e Paris 13, CNRS UMR 7539,
Laboratoire Analyse, G\'eom\'etrie et Applications 93430 Villetaneuse, France}

\email{phanqh@math.univ-paris13.fr}


\begin{abstract}
We consider the  Hardy-H\'enon parabolic equation $u_t-\Delta u =|x|^a |u|^{p-1}u$ with $p>1$ and $a\in {\mathbb R}$. We establish the space-time singularity and decay estimates, and Liouville-type theorems for radial and nonradial solutions. As applications, we  study universal and a priori bound of global solutions as well as the blow-up estimates for the corresponding initial boundary value problem.
\end{abstract}

\maketitle

\section{Introduction}
In this paper, we study the semilinear parabolic equation of the form 
\begin{align}\label{1}
 u_t-\Delta u =|x|^a |u|^{p-1}u, \quad (x,t)\in {\Omega\times I}
\end{align}
where $\Omega$ is a domain of ${\mathbb R}^N$, $p>1$, and $I$ is an interval of $\mathbb R$. We assume throughout that $a>-2$ when $N\geq 2$, and  $a> -1$ when $N=1$.

Throughout this paper, unless otherwise specified, solutions are considered in the class
\begin{equation}\label{solclass}
\begin{cases}
C^{2,1}(\Omega\times I), &\hbox{if $a\geq 0$},  \\
\noalign{\vskip 1mm}
C^{2,1}(\Omega\setminus\{0\}\times I)\cap C^{0,0}(\Omega\times I), &\hbox{if $a<0$},
\end{cases}
\end{equation}
and are assumed to satisfy the equation pointwise, except at $x=0$ if $a<0$ and $0\in \Omega$. This choice is natural since we are primarily interested in classical solutions,
except for possible singularity at the origin if $a<0$ and $0\in\Omega$. For $N=1$ (and $-1<a<0$ and $0\in\Omega$), we instead consider distributional solutions which belong to $C^{0,0}(\Omega\times I)$.

The restriction $a>-2$ when $N\geq 2$ is reasonable due to the regularity at the origin of stationary solutions  (cf. \cite[Lemma 6.2]{BV99}, \cite{ DP, GS81a}). In this case,  it turns out that any (classical) solution in the sense (\ref{solclass}) is also a distributional solution (see Lemma~\ref{lemappe1} in Appendix).
The case $N=1$ is more peculiar -- see Proposition~\ref{reappe} and the preceding paragraph.

For the statement of  main results, let us introduce the following exponents: 
\begin{align}
p_S(a):=\begin{cases}
\frac{N+2+2a}{N-2}\quad &\text{if}\quad N\geq 3,\\
\infty \quad &\text{if}\quad N=1,2,
\end{cases}
\end{align}
$p_S:=p_S(0)$ and
\begin{align}
 p_B:=\begin{cases}
\frac{N(N+2)}{(N-1)^2}\quad &\text{if}\quad N\geq 2,\\
\infty \quad &\text{if}\quad N=1.
\end{cases}
\end{align}
\subsection{Liouville-type theorems}
As the first topic, we are interested in the Liouville property -- i.e. the nonexistence of solution of problem (\ref{1})
in the entire space ${\mathbb R}^N\times {\mathbb R}$.
We first recall its elliptic counterpart 
\begin{align}\label{2}
 -\Delta u =|x|^a |u|^{p-1}u, \quad x\in {\mathbb R}^N.
\end{align}
The Liouville-type  result for (\ref{2})  plays an important role in the parabolic problem but it is not completely solved. For radial solutions, the problem (\ref{2}) has no positive radial solution if and only if $p< p_S(a)$ and it has been conjectured that the nonexistence of positive solution holds under that condition.  However, the Liouville-type result for  (\ref{2}) was only proved under stronger assumption, namely $p<\min\{p_S, p_S(a)\}$, which is not optimal when $a>0$. Recently, the conjecture was shown in \cite{PhS} for  bounded positive solution in dimension $N=3$.
\medskip

For corresponding parabolic equation, the Liouville property has been studied in special case $a=0$ for  nonnegative and nodal radial solutions (see \cite{BPQ11, BV98, PQ06,  PQS07b}).  The following results are known to be true.

\medskip
\par\noindent {\bf Theorem A.}
{\it  
\smallskip
\par\noindent (i) Let $a=0$ and $1<p<p_S$. Then equation (\ref{1}) has no nontrivial nonnegative radial solution in ${\mathbb R}^N\times {\mathbb R}$.
\smallskip
\par\noindent (ii) Let $a=0$ and $1<p<p_B$. Then equation (\ref{1}) has no nontrivial nonnegative  solution in ${\mathbb R}^N\times {\mathbb R}$.
}

\medskip
\par\noindent {\bf Theorem B.}
{\it  
\smallskip
\par\noindent (i) Let $a=0$, $1<p<p_S$ and let $u=u(r,t)$ be a classical radial solution of (\ref{1}) in ${\mathbb R}^N\times {\mathbb R}$ with the number of sign-changes satisfying
$$z_{(0,\infty)}(u(t))\leq M,\quad \forall t\in {\mathbb R}.$$
Then $u\equiv 0$.
\smallskip
\par\noindent (ii) Let $a=0$, $N=1$ and let $u=u(x,t)$ be a classical solution of (\ref{1}) in ${\mathbb R}\times {\mathbb R}$ with the number of sign-changes satisfying
$$z_{\mathbb R}(u(t))\leq M,\quad \forall t\in {\mathbb R}.$$
Then $u\equiv 0$.
}

\medskip
Theorem~A was shown in \cite{BV98, PQS07b, PQ06}, and Theorem~B is recently proved  in \cite{BPQ11}. The upper  bound of exponent $p$ in Theorem~A(i) and in Theorem~B(i) is optimal due to the existence of positive (bounded) radial solution of $-\Delta u=|u|^{p-1}u$ in ${\mathbb R}^N$ for $p\geq p_S$.

For case $a\ne 0$, the Liouville property  is much less understood even for radial solution. Up to now, the only available result of this kind is the Fujita-type (see \cite{Pin97}, or \cite[section 26]{MP01}), which states there is no positive solution in ${\mathbb R}^N\times \mathbb R_+$ if and only if $1<p\leq 1+\frac{2+a}{N}$. In this paper, we will establish  Liouville-type theorems in case $a\ne 0$ for a larger range of $p$. We have the following results.
\begin{theorem}\label{th1a}
\smallskip
\par\noindent (i) Let $1<p<\min\{p_B,p_S(a)\}$ and $u$ be bounded nonnegative solution of equation (\ref{1}) in ${\mathbb R}^N\times \mathbb R$. Then   $u\equiv 0$.

\smallskip

\par\noindent (ii) Let  $1<p<p_S(a)$ and $u$ be bounded nonnegative radial solution of equation (\ref{1}) in ${\mathbb R}^N\times \mathbb R$. Then   $u\equiv 0$.
\end{theorem}

For  sign-changing solution, let us recall the definition of zero number. Given an open inteval $I\subset \mathbb R$ and $v\in C(I)$, then the zero number of $v$ in $I$ is defined by
$$z_I(v):=\sup\{j: \exists x_1,..., x_{j+1}\in I,  x_1<x_2<...<x_{j+1}, v(x_i)v(x_{i+1})<0, \text{ for } i=1,...,j \}.$$ 
We have the following result.
\begin{theorem}\label{th1b}
\par\noindent  Let  $1<p<p_S(a)$ and let $u=u(r,t)$ be a  radial solution of (\ref{1}) in ${\mathbb R}^N\times \mathbb R$ with the number of sign-changes satisfying
$$z_{(0,\infty)}(u(t))\leq M,\quad \forall t\in {\mathbb R}.$$
Then $u\equiv 0$.
\end{theorem}
The proofs of Theorem~\ref{th1a} and Theorem~\ref{th1b} follow the idea as in \cite{BPQ11, AS10, QS10}, which consists of three steps :
\begin{enumerate}
\item Showing spatial decay of solutions (see  Theorem~\ref{th2}(ii) and Theorem~\ref{th3}(ii) below). 
 \item Using the Lyapunov functional and decay estimate  of solutions to show that both $\alpha$- and $\omega$-limit sets of any solution are nonempty and consist of equilibria.
\item Combining with the nonexistence of nontrivial equilibria to have the contradiction. 
\end{enumerate}
\begin{remark}{\rm
(a) We note that the condition $p<p_S(a)$ in Theorem~\ref{th1a} (ii) and Theorem~\ref{th1b}  is optimal, due to the existence of bounded positive  radial solution of $-\Delta u=|x|^au^p$ in ${\mathbb R}^N$ for $p\geq p_S(a)$.

(b) Theorem~\ref{th1a}(ii) for $a>0$ can be proved by another, completely different method , namely intersection-comparison argument (see \cite{PQ06}). For this case, the proof is totally similar to that in \cite{PQ06}.

(c) Related to Theorem~\ref{th1a}, it is a natural conjecture that the nonexistence of entire nonnegative nontrivial solution holds for $p<p_S(a)$. However, it seems still difficult, even for special case $a=0$.}
\end{remark}
\medskip
\subsection{Singularity and decay estimates}
As the next topic, we establish the space-time singularity and decay estimates of  solutions of equation (\ref{1}). The following theorem is a parabolic counterpart of \cite[Theorem 1.2]{PhS}. The similar results for $a=0$ has been proved in \cite[Theorem 3.1]{PQS07b}.
\begin{theorem}\label{th2}
 (i) Let $u$ be a nonnegative solution of (\ref{1}) on $\Omega\times (0,T)$ where $\Omega=\{0<|x|<\rho\}$. Assume  that either
\begin{align}\label{condp}
 p<p_B, \qquad \text{ or }  \text{ $u$ is radial}.
\end{align}
Then for all $0<|x|<\rho/2$ and $t\in (0,T)$, there holds
\begin{align}\label{singul1}
|x|^{a/(p-1)}u(x,t)+\big||x|^{a/(p-1)}\nabla u(x,t)\big|^{2/(p+1)}\leq C\left (t^{-1/(p-1)}+ (T-t)^{-1/(p-1)}+ |x|^{-2/(p-1)}\right), 
\end{align}
where $C=C(N,p,a)$. 

(ii) Let $u$ be a nonnegative solution of (\ref{1}) in $\Omega\times (0,T)$ where $\Omega=\{|x|>\rho\}$. Assume  that either
\begin{align*}
 p<p_B,  \qquad\text{ or }  \text{ $u$ is radial}.
\end{align*}
Then for all $|x|>2\rho$ and $t\in (0,T)$, there holds
\begin{align}\label{singul2}
|x|^{a/(p-1)}u(x,t)+\big||x|^{a/(p-1)}\nabla u(x,t)\big|^{2/(p+1)}\leq C\left (t^{-1/(p-1)}+ (T-t)^{-1/(p-1)}+ |x|^{-2/(p-1)}\right), 
\end{align}
where $C=C(N,p,a)$.
\end{theorem}
 
For sign-changing solution, we have the following. We stress that there is no restriction on the upper bound of exponent $p$.
\begin{theorem}\label{th3}
 (i)  Let $u=u(r,t)$ be a radial solution  of (\ref{1})  on $\Omega\times (0,T)$ where $\Omega=\{0<|x|=r<\rho\}$ with the number of sign-changes satisfying
$$z_{(0,\rho)}(u(t))\leq M,\quad \forall t\in (0,T).$$
Then for all $0<r<\rho/2 $ and $t\in (0,T)$, there holds
\begin{align*}
r^{a/(p-1)}|u(r,t)|+\big|r^{a/(p-1)} u_r(r,t)\big|^{2/(p+1)}\leq C\left (t^{-1/(p-1)}+ (T-t)^{-1/(p-1)}+ r^{-2/(p-1)}\right)
\end{align*}
where $C=C(N,p,a,M)$. 

\smallskip
(ii) Let $u=u(r,t)$ be a radial solution in $\Omega\times (0,T)$ where $\Omega=\{|x|=r>\rho\}$ with the number of sign-changes satisfying
$$z_{(\rho, \infty)}(u(t))\leq M,\quad \forall t\in (0,T).$$
Then for all $r>2\rho$ and $t\in (0,T)$, there holds
\begin{align*}
r^{a/(p-1)}|u(r,t)|+\big|r^{a/(p-1)} u_r(r,t)\big|^{2/(p+1)}\leq C\left (t^{-1/(p-1)}+ (T-t)^{-1/(p-1)}+ r^{-2/(p-1)}\right)
\end{align*}
where $C=C(N,p,a,M)$.
\end{theorem}

 The proofs of Theorem~\ref{th2} and  Theorem~\ref{th3} rely on:
\begin{enumerate}
\item a change of variable, that allows to replace the coefficient $|x|^a$ with a smooth function 
which is bounded and bounded away from $0$ in a suitable spatial domain;

\item  a generalization of a doubling-rescaling argument from \cite{PQS07} (see Lemma~\ref{lem3a} below).

\item The corresponding  Liouville-type theorem for equation (1) with $a=0$.
\end{enumerate}
\medskip
 \begin{remark}\label{rm1}{\rm
(a) The estimates in Theorerm~\ref{th3} in case $a=0$ give a similar form as in \cite[Corollary 3.2]{MM04} and \cite[Proposition 2.5 and 2.7]{MM09}  for radial solutions of supercritical nonlinear heat equation. As an improvement, the constants $C$ are here universal, but at expense of further restriction on finite number of sign-changes of solutions. Our argument is based on rescaling and doubling property while that one in \cite{MM04, MM09} is based on energy estimates.

(b) If we replace the interval $(0,T)$ by $\mathbb R$ in Theorem~\ref{th2}(ii) and  in Theorem~\ref{th3}(ii), then we have the  spatial decay estimate 
\begin{align*}|u(x,t)|\leq C|x|^{-(2+a)/(p-1)},  \quad |\nabla u(x,t)|\leq C|x|^{-(p+1+a)/(p-1)},\quad  |x|>0,\; t\in \mathbb R. 
 \end{align*}
This is an important feature that will be used in proof of Theorem~\ref{th1a} and Theorem~\ref{th1b}. }
 \end{remark}

\subsection{A priori bound  of global solutions and blow-up estimates}
As applications of Liouville-type results, let us consider the corresponding initial-boundary value problem:
\begin{align}\label{BVP}
\begin{cases}
 u_t-\Delta u=|x|^a u^p, \quad & x\in \Omega,\; 0<t<T,\\
u=0, &x\in \partial \Omega, \; 0<t<T,\\
u(x,0)=u_0(x), \quad & x\in \Omega.
\end{cases}
\end{align}
where $\Omega$ is a smooth bounded domain in ${\mathbb R}^N$ and contains the origin. We have a priori bound of nonnegative solutions as follows.

\begin{theorem}\label{th5}
Let  $1<p<\min(p_S, p_S(a))$. Assume $u$ is any global solution of problem (\ref{BVP}) with initial data $u_0\geq 0$. Then 
\begin{align}\label{z4}
 \sup_{t\geq 0}\|u(t)\|_{\infty}\leq C( \|u_0\|_{\infty}).
\end{align}
Moreover, if $\Omega=B_R$ and $u_0$ is radial then (\ref{z4}) still holds whenever $1<p< p_S(a)$.
\end{theorem}

\begin{remark}{\rm
 We recall that a priori bound of nonnegative solutions of elliptic problem $-\Delta u=|x|^au^p$ has been proved  under the condition $p<\min(p_S, p_S(a))$ (see \cite[Theorem 1.3]{PhS}).   Theorem~\ref{th5}  say that a priori bound (\ref{z4}) for parabolic countepart also holds under this condition. In special case $a=0$, such a priori bound was proved by Giga \cite{Gig86} for nonnegative solutions, and by Quittner \cite{Qui99} for sign-changing solutions.} 
\end{remark}
\medskip

We next give results of universal initial and final time blow-up rates. The similar result for case $a=0$ has been proved in \cite{PQS07b}. The final time blow-up estimate of problem (\ref{BVP}) was  estalished in \cite[Theorem 1.2 and 1.3]{AT05}, under a stronger condition $1<p<1+\min\{2/N, (2+a)/N\}$.
\begin{theorem}\label{th4}
  Let $u$ be a positive solution of (\ref{BVP}). Assume  that either
\begin{align}
 p<\min\{p_B,p_S(a)\},  \qquad \text{ or } p<p_S(a), \text{ $\Omega$ is a ball $B_R$ and $u$ is radial}.
\end{align}

(i) If $T<\infty$ then there holds
\begin{align}\label{universalbound}
u(x,t)\leq C\left (1+ t^{-1/(p-1)}+ (T-t)^{-1/(p-1)} \right),\quad x\in \Omega,\; 0<t<T,
\end{align}
where $C=C(\Omega,p,a)$. 

(ii) If $u$ is global then there holds
\begin{align}
u(x,t)\leq C\left (1+ t^{-1/(p-1)}\right),\quad x\in \Omega,\; t>0, 
\end{align}
where $C=C(\Omega,p,a)$. 
\end{theorem}

Theorem~\ref{th4}(ii) in particular implies universal bounds, away from $t=0$, for all global solutions of problem (\ref{BVP}). In last result, we provide such bounds under different assumptions on $p,a$ and $N$. This result gives a less precise conclusion than that in Theorem~\ref{th4}(ii) but it can be applied in a different range of parameters, due to a completely different method. Whereas Theorem~\ref{th4} relies on Liouville theorems and doubling arguments, the method of proof of Theorem~\ref{th6} is different, based on a combination of energy and rescaling arguments (see \cite{Qui01, QSW04}).

\begin{theorem}\label{th6}
Let $a>0$, $N\leq 4$,  and $1<p<\frac{N+2+a}{N-2+a}$ ($1<p<\infty$ when $N=1$). Then for all $\tau >0$, there exists $C=C(\Omega, p, a, \tau)$ such that any nonnegative global solution of problem (\ref{BVP}) satisfies
\begin{align}\label{universal}
 \sup_{t\geq \tau}\|u(t)\|_\infty\leq C(\Omega, p, a, \tau).
\end{align}
\end{theorem}

\medskip
 In this paper, the proofs of Theorem~\ref{th2}-\ref{th6}  all make use of rescaling techniques,   combined with some additional arguments, such as, doubling properties, parabolic Liouville-type theorems (for both case $a=0$ and $a\ne 0$), or energy arguments. The classical rescaling argument was first introduced by Gidas and Spruck (\cite{GS81b}) for elliptic problem, it was then significantly  improved in \cite{Gig86, Hu96, PQS07, PQS07b} for elliptic and parabolic problems. In particular, the authors in \cite{PQS07, PQS07b} have developped the doubling property (which is an extension of an idea of \cite{Hu96}) that enables one to obtain a variety of important results such as: singularity and decay estimates, a priori bound and universal bounds of solutions, etc.... We essentially employ this powerful idea and introduce some new rescalings to deal with some new difficulties arising due to the degeneracy and singularity of the term $|x|^a$. We intend to provide the details of various rescaling arguments in order to make precise the differences among casses.

\medskip
We close the introduction by mentioning other work related to problem (\ref{BVP}). The Cauchy problem corresponding to problem (\ref{BVP}) (i.e. $\Omega={\mathbb R}^n$) has been widely studied, and the existence and nonexistence of global solution were established \cite{Pin97, Hir08, FT00}. The asymptotics, stabilization and blow-up phenomenon of the Cauchy problem are considered in  \cite{Wan93, DLY06, LL08}.  The blow-up phenomenon for initial-boundary value problem (\ref{BVP}) can be found in  \cite{GS11, GLS10}, where the authors constructed  a special solution that blows up at the origin, and also gave some sufficient conditions that ensure the origin is not a blow-up point.

The organization of this paper is as follows. Section 2 contains the proof of the singularity and decay estimates (Theorem~\ref{th2} and \ref{th3}). Section 3 contains the proof of Liouville-type theorems. In Section 4, we give the proofs of Theorems~\ref{th5}-\ref{th6}.

\section{Singularity and decay estimates}
In this Section, we give a relatively simple proof of Theorem~\ref{th2}. Theorem~\ref{th3} can be proved by the same argument. 
We need the following lemma.

\begin{lemma}\label{lem3a}
Let $\mathcal C=\{x\in {\mathbb R}^{N}: 1<|x|<2\}$,  
$\alpha\in (0,1]$ and $c\in C^\alpha(\overline{\mathcal C})$ be a function  satisfying
\begin{equation}\label{boundc}
\|c\|_{C^\alpha(\overline{\mathcal C})}\leq C_1\qquad \text{ and } \qquad c(x)\geq C_2,\quad x\in \overline{\mathcal C},
\end{equation}
for some constants $C_1, C_2>0$.  Let $u$ be positive classical solution of
\begin{equation}\label{eqcup}
u_t-\Delta u=c(x)u^p,\quad (x,t)\in \mathcal C\times {\mathbb R}.
\end{equation}
Assume  that either
\begin{align}
 p<p_B,  \qquad\text{ or }  \text{ $c, u$ are radial}.
\end{align}
Then there exists a constant $C=C(\alpha, C_1, C_2, p, N)$, such that, for all $(x,t)\in \mathcal C\times(0,T)$, there holds
\begin{align}\label{est}
|u(x,t)|+|\nabla u(x,t)|^{2/(p+1)}\leq C\bigl(1+t^{-1/(p-1)}+(T-t)^{-1/(p-1)}+{\rm dist}^{-2/(p-1)}(x,\partial \mathcal C)\bigr).
 \end{align}
\end{lemma}

\begin{proof} 
We follow the argument in \cite{PQS07b}, we denote the parabolic distance
\begin{align}\label{dist}
d_P((x,t),(y,s)):=|x-y|+|t-s|^{1/2}.
\end{align}
Let $D=\mathcal C\times (0,T)\in {\mathbb R}^{N+1}$ then the estimate (\ref{est}) can be written as 
\begin{align}
u(x,t)+|\nabla u(x,t)|^{2/(p+1)}\leq C\bigl(1+d^{-2/(p-1)}((x,t),\partial D)\bigr),\; (x,t)\in D
\end{align}

Arguing by contradiction, we suppose that there exist sequences $c_k, u_k, T_k$
verifying (\ref{boundc}), (\ref{eqcup}) and points $(y_k,\tau_k)$, such that the functions 
$$M_k=|u_k|^{(p-1)/2}+|\nabla u_k|^{(p-1)/(p+1)}$$
satisfy
$$M_k(y_k,\tau_k)>2k\bigl(1+d_P^{-1}((y_k,\tau_k),\partial D_k)\bigr)> 2k\,d_P^{-1}((y_k,\tau_k),\partial D_k\bigr),\quad  D_k=\mathcal C\times (0,T_k).$$
By the doubling lemma in \cite[Lemma 5.1]{PQS07} with $X={\mathbb R}^{N+1}$, equipped with parabolic distance $d_P$, there exists $(x_k,t_k)\in D_k$ such that
$$M_k(x_k,t_k)\geq M_k(y_k,\tau_k),\quad M_k(x_k,t_k)>2k\,d_P^{-1}((x_k,t_k),\partial D_k),$$
and 
\begin{align}
M_k(x,t)\leq &2M_k(x_k,t_k), \quad\hbox{ for all $(x,t)$ such that } d_P((x,t),(x_k,t_k))\leq kM_k^{-1}(x_k,t_k).\label{ineqMk}
\end{align}
We note that $(x,t)$ satisfying (\ref{ineqMk}) is automatically contained in $D_k$.

We have 
\begin{equation}\label{convlambda}
\lambda_k:=M_k^{-1}(x_k,t_k)\to 0 \quad \text{as }  k\to\infty,
\end{equation}
due to $M_k(x_k,t_k)\geq M_k(y_k,\tau_k)>2k$. 

We now consider the nonradial and radial cases separately.

{\it A. The nonradial case. } Let 
$$v_k(y,s)=\lambda_k^{2/(p-1)}u_k(x_k+\lambda_k y,t_k+\lambda_k^2 s),\qquad \tilde c_k(y)=c_k(x_k+\lambda_k y), \quad (y,s)\in \tilde{D}_k,$$
where 
$$\tilde{D}_k=\{|y|<k/2\}\times\{|s|<k^2/4\}.$$
We note that $|v_k|^{(p-1)/2}(0,0)+|\nabla v_k|^{(p-1)/(p+1)}(0,0)=1$, 
\begin{equation}\label{boundvk}
\left[|v_k|^{(p-1)/2}+|\nabla v_k|^{(p-1)/(p+1)}\right](y,s)\leq 2, \quad (y,s)\in \tilde{D}_k,
\end{equation}
due to (\ref{ineqMk}), and we see that $v_k$ satisfies
\begin{equation}\label{eqnvk}
\partial_s v_k-\Delta v_k=\tilde c_k(y)v_k^p, \quad (y,s)\in \tilde{D}_k.
\end{equation}

On the other hand, due to (\ref{boundc}), we have $C_2\leq \tilde c_k\leq C_1$ and, for each $R>0$ and $k\geq k_0(R)$ large enough,
\begin{equation}\label{boundAscolic}
|\tilde c_k(y)-\tilde c_k(z)|\leq C_1|\lambda_{k}(y-z)|^\alpha\leq C_1|y-z|^\alpha,\quad |y|,|z|\leq R.
\end{equation}
Therefore, by Ascoli's theorem, there exists $\tilde c$ in $C({\mathbb R}^N)$, with $\tilde c\geq C_2$ such that, 
after extracting a subsequence,
$\tilde c_k\to \tilde c$ in $C_{loc}({\mathbb R}^N)$. 
Moreover, (\ref{boundAscolic}) and (\ref{convlambda}) imply that $|\tilde c_k(y)-\tilde c_k(z)|\to 0$ as $k\to\infty$, so that the function $\tilde c$ is actually a constant $C>0$.

Now, for each $R>0$ and $1<q<\infty$, by (\ref{eqnvk}), (\ref{boundvk}) and interior parabolic $L^q$ estimates, 
the sequence $v_k$ is uniformly bounded in $W^{2,1}_q(B_R\times (-R, R))$.  
Using standard imbeddings and interior parabolic $L^q$ estimates, after extracting a subsequence,  
we may assume that $v_k\to v$ in $C^{2,1}_{loc}({\mathbb R}^N\times {\mathbb R})$.
It follows that $v\geq 0$ is a classical solution of
$$v_s-\Delta v= Cv^p, \quad (y,s)\in  {\mathbb R}^N\times {\mathbb R},$$
and $|v|^{(p-1)/2}(0,0)+|\nabla v|^{(p-1)/(p+1)}(0,0)=1$.
Since $p<p_B$, this contradicts Theorem A(ii).

\medskip
{\it B. The radial case. } Since $u_k, c$ are radial, we write $u_k=u_k(r,t)$, $c_k=c_k(r)$ and $M_k=M(r,t)$, where $r=|x|$. Then $u_k$ solves the equation
\begin{align*}
 u_t-u_{rr}-\frac{N-1}{r}u_r=c_k(r)u^p.
\end{align*}
Assume that $|x_k|=r_k$, it follows from (\ref{ineqMk}) that 
\begin{equation*}
M_k(r,t)\leq 2M_k(r_k,t_k), \quad\hbox{ for all $(r,t)$ such that } |r-r_k|+|t-t_k|^{1/2}\leq k\lambda_k:=kM^{-1}_k(r_k,t_k).
\end{equation*}
We rescale by
\begin{align*}
 v_k(\rho,s):=\lambda_k^{2/(p-1)}u_k(r_k+\lambda_k\rho, t_k+\lambda_k^2 s), \quad (\rho, s)\in \tilde{D}_k,
\end{align*}
where
$$\tilde{D}_k=\big(-\min(r_k/\lambda_k,\, k/2),\; k/2\big)\times(-k^2/4, k^2/4).$$
Then $v_k$ solves the equation
\begin{align*}
 v_s-v_{\rho\rho}-\frac{N-1}{\rho+r_k/\lambda_k}v_\rho=\tilde{c_k}v^p, \quad \tilde{c_k}(\rho)=c_k(r_k+\lambda_k\rho), \quad (\rho,s)\in \tilde{D}_k,
\end{align*}
and we note that $|v_k|^{(p-1)/2}(0,0)+|\nabla v_k|^{(p-1)/(p+1)}(0,0)=1$, 
\begin{equation*}
\left[|v_k|^{(p-1)/2}+|\nabla v_k|^{(p-1)/(p+1)}\right](\rho,s)\leq 2, \quad (\rho,s)\in \tilde{D}_k.
\end{equation*}
Similar to the nonradial case,  after extracting a subsequence,  we may assume that $\tilde {c}_k(\rho)\to C$ in $C_{loc}({\mathbb R})$. Since $r_k\in (1,2)$  then $r_k/\lambda_k \to \infty$. Passing to the limit, we obtain a nonnegative bounded solution $v$ of the equation
$$v_s-v_{\rho\rho}=Cv^p \quad \text{ in } \quad {\mathbb R}\times {\mathbb R}$$
and
$$|v|^{(p-1)/2}(0,0)+|\nabla v|^{(p-1)/(p+1)}(0,0)=1.$$
This contradicts Theorem A(ii) for $N=1$ and concludes the proof.
\end{proof}
{\it Proof of Theorem~\ref{th2}.} 
Assume either $\Omega=\{x\in{\mathbb R}^N;\, 0<|x|<\rho\}$ and $0<|x_0|<\rho/2$,
or $\Omega=\{x\in{\mathbb R}^N;\, |x|>\rho\}$ and $|x_0|>2\rho$.
Let $R=\textstyle\frac23\, |x_0|$ and we denote
\begin{align*}
U(y,s)=R^{\frac{2+a}{p-1}}u(Ry, R^2s).
\end{align*}
Then $U$ is a solution of
\begin{align*}
U_s-\Delta U=c(y)U^p,\ \ (y,s)\in \mathcal C\times (0, R^{-2}T),\quad\hbox{ with } c(y)=|y|^a,\;\; \mathcal C=\{y\in {\mathbb R}^N: 1<|y|<2\}.
\end{align*}
Notice that $|y| \in [1,2]$ for all $y\in \overline{\mathcal C} $.
Moreover $\|c\|_{C^1(\overline{\mathcal C})}\leq C(a)$.
Then applying Lemma~\ref{lem3a}, we have 
$$U(R^{-1}x_0,R^{-2}t)+|\nabla U(R^{-1}x_0,R^{-2}t)|^{2/(p+1)}\leq C\bigl(1+(R^{-2}t)^{-1/(p-1)}+(R^{-2}T-R^{-2}t)^{-1/(p-1)}\bigr),$$
 hence
$$R^{a/(p-1)}u(x_0,t)+ \big|R^{a/(p-1)}\nabla u(x_0,t)\big|^{2/(p+1)}\leq C \bigl(t^{-1/(p-1)}+(T-t)^{-1/(p-1)}+R^{-2/(p-1)}\bigr),$$
which yields the desired conclusion. 
\qed

\medskip
{\it Proof of Theorem~\ref{th3}.} By the similar argument,   we have the same estimate (\ref{est}) for radial solution with finite number of sign-changes. The only  thing taken into consideration is that, in the last step of proof of Lemma~\ref{lem3a}, we have a contradiction with Liouville-type theorem for nodal solution (see \cite[Theorem 1.4]{BPQ11}). The rest of proof is similar.\qed

\section{Liouville type theorem}
In this section, we will only prove Theorem~\ref{th1a}. And by this method, one can prove Theorem~\ref{th1b} similarly.
The proof is based on properties of energy functional. We note that the solutions in the case of the whole space ${\mathbb R}^N$ need not a priori belong to the energy space. However, as shown in the following lemma, this will turn out to be true thanks to the spatial decay estimates in (\ref{decay1}). Moreover the case $a<0$ is more delicate and requires additional arguments.

For $u$ solution of equation (\ref{1}) in ${{\mathbb R}^N}\times {\mathbb R}$, we denote (formally) the energy functional 
\begin{align}\label{energyfun}
 E(t):=\frac12\int_{{\mathbb R}^N}|\nabla u(t)|^2dx-\frac{1}{p+1}\int_{{\mathbb R}^N}|x|^au^{p+1}(t)dx.
\end{align}

\begin{lemma}\label{energy}
Assume $p<p_S(a)$ and (\ref{condp}). For any solution $u$ of equation  (\ref{1}) in ${{\mathbb R}^N}\times {\mathbb R}$, the energy functional (\ref{energyfun}) is well defined for any $t\in {\mathbb R}$. Moreover, for any $t_1<t_2$, we have 
\begin{align}\label{energyine}
 \int_{t_1}^{t_2}\int_{{\mathbb R}^N}u_t^2(t)dxdt\leq -E(t_2)+E(t_1).
\end{align}
  \end{lemma}
\begin{proof}
By Theorem~\ref{th2} (see Remark~\ref{rm1}(b)), $u$ has spatial decay estimates 
\begin{align}\label{decay1}
u(x,t) \leq C |x|^{-\frac{2+a}{p-1}}, \qquad |\nabla u(x,t)| \leq C |x|^{-1-\frac{2+a}{p-1}}, \quad |x|>1. 
\end{align}
We first show that 
\begin{align}
 \label{decay2}
|u_t(x,t)|\leq C |x|^{-2-\frac{2+a}{p-1}},  \quad |x|>2.
\end{align}
Indeed, for any $R>2$, let $U(y,s)=R^{(2+a)/(p-1)}u(Ry, R^2s)$ then $U(y,s)\leq C$ in $(B_4\setminus B_{1/2})\times \mathbb R$, where $C =C(N,p,a)$, and 
\begin{align*}
 U_s-\Delta U=|y|^a U^p, \quad (y,s)\in (B_4\setminus B_{1/2})\times \mathbb R
\end{align*} 
It follows from  boostrap argument for parabolic regularity that $|U_s(y,s)|\leq C$ for all $(y,s)\in  (B_2\setminus B_{1})\times \mathbb R$, where $C$ does not depend on $R$. Hence,
$$|u_t(x,t)| \leq CR^{-2-(2+a)/(p-1)} \leq C|x|^{-2-(2+a)/(p-1)}, \quad (x,t) \in (B_{2R}\setminus B_{R})\times \mathbb R,$$
and (\ref{decay2}) follows.

Combining these decay estimates with  $p<p_S(a)$, we have for any $t\in {\mathbb R}$, 
$$|\nabla u(t)|^2 \in L^1({\mathbb R}^N\setminus B_1),\; \;u_t^2(t) \in L^1({\mathbb R}^N\setminus B_1), \;\; \quad |x|^au^{p+1}(t)\in L^1({\mathbb R}^N).$$
Hence, if $a\geq 0$, since there is no singularity at $x=0$,  the energy functional (\ref{energyfun}) is well defined and (\ref{energyine}) holds since 
\begin{align*}
 \frac{dE(t)}{dt}=-\|u_t(t)\|_2^2.
\end{align*} 

We now consider case $a<0$. Since the term $f=|x|^au^{p}\in L^\infty_{loc}({\mathbb R}; L^{\tilde{q}}(B_2))$ for any $1<\tilde{q}<N/|a|$, using the cut-off function and variation-of-constants formula (obtained by Lemma~\ref{lemappe2} in Appendix), we have
$$u\in L^\infty_{loc}({\mathbb R}; W^{2-\delta, \tilde{q}}(B_1)), \quad 1<\tilde{q}<N/|a|.$$
Choose $\delta>0$ small enough such that $W^{2-\delta, \tilde{q}}(B_1)$ is continuously embedded into $W^{1,2}(B_1)$. Hence, $|\nabla u(t)|^2 \in L^1({\mathbb R}^N)$ and the energy functional (\ref{energyfun}) is well defined. To prove (\ref{energyine}),  we may assume that $t_1=0, t_2>0$, we consider the following problem
\begin{align*}
\begin{cases} \partial_t v_\varepsilon-\Delta v_\varepsilon=(|x|+\varepsilon)^a v_\varepsilon^p,  \quad (x,t)\in {\mathbb R}^N\times [0, t_2],\\
v(x,0)=u(x,0).
\end{cases}
\end{align*}
Note that $a<0$, by comparison property we have $v_{\varepsilon}$ is increasing and $0<v_\varepsilon\leq u$. This implies in particular that $v_\varepsilon$ satisfies the first part of spatial estimate (\ref{decay1}). Let us show spatial decay of $\nabla v_\varepsilon$.
For any $R>2$, let $V(y,s)=R^{(2+a)/(p-1)}v_\varepsilon(Ry, R^2s)$, then $V(y,s)\leq C$ in $(B_4\setminus B_{1/2})\times \mathbb R$, where $C =C(N,p,a)$, and 
\begin{align*}
 V_s-\Delta V=(|y|+\varepsilon/R)^a V^p, \quad (y,s)\in (B_4\setminus B_{1/2})\times \mathbb R
\end{align*}
For any $\varepsilon\leq 1/2$, we have $1/4<|y|+(\varepsilon/R)<5$ for all $1/2<|y|<4$. The parabolic estimates imply $|\nabla V(y,s)|\leq C$ for all $(y,s)\in (B_2\setminus B_{1})\times \mathbb R$, where $C$ does not depend on $R$. Hence, $|\nabla v_\varepsilon(x,t)|\leq CR^{-1-(2+a)/(p-1)}$ for all $(x,t)\in (B_{2R}\setminus B_{R})\times \mathbb R$. Therefore, $|\nabla v_\varepsilon(x,t)|\leq C|x|^{-1-(2+a)/(p-1)}$ for any $\varepsilon<1/2$ and $|x|>2$.

 Let $v=\lim v_{\varepsilon}$ and $e^{t\Delta}$ denote the heat semigroup in ${\mathbb R}^N$, we show that $v=u$. Indeed, by the variation-of-constants formula, we deduce that
\begin{align*}
 u(t)-v(t)=\int_{0}^{t}e^{(t-s)\Delta}(|.|^a(u^p-v^p))ds=\int_{0}^{t}e^{(t-s)\Delta}(|.|^aH(u,v)(u-v))ds,
\end{align*}
where $0\leq H(u,v)\leq pu^{p-1}$. Let $w=u-v$ then $w(0)=0$ and
\begin{align*}
 \|w(t)\|_\infty\leq \int_{0}^{t} (t-s)^{-N/(2q)}\|w(s)\|_\infty \||x|^aH(u,v)\|_qds.
\end{align*}
We choose $q=1$ when $N=1$, and $q=N/(|a|+\gamma)$ when $N\geq 2$, where $\gamma>0$ satisfies $|a|+\gamma< 2$. Then $N/(2q)<1$ and 
\begin{align*}
 \||x|^aH(u,v)\|_q\leq  \|p|x|^a u^{p-1}\|_q  \leq C.
\end{align*}
Hence, 
\begin{align*}
 \|w(t)\|_\infty\leq C\int_{0}^{t} (t-s)^{-N/(2q)}\|w(s)\|_\infty ds.
\end{align*}
Consequently, $w\equiv 0$, or $u\equiv v$.

Let us denote by $E_\varepsilon(t)$ the energy functional with respect to $v_\varepsilon$, which is well-defined due to the spatial decay of $v_\varepsilon$ and $\nabla v_\varepsilon$. Then we have
\begin{align}\label{appro}
 \int_{0}^{t_2}\int_{{\mathbb R}^N}|\partial_t v_\varepsilon|^2(t)dxdt = -E_\varepsilon(t_2)+E(0).
\end{align}
Hence
\begin{align*}
 \int_{0}^{t_2}\int_{{\mathbb R}^N}|\partial_t v|^2(t)dxdt&\leq \liminf\limits_{\varepsilon\to 0}\int_{0}^{t_2}\int_{{\mathbb R}^N}|\partial_t v_\varepsilon|^2(t)dxdt\\
&\leq \liminf\limits_{\varepsilon\to 0}\bigg(-\frac{1}{2}\int_{{\mathbb R}^N}|\nabla v_\varepsilon (t_2)|^2dx+\frac{1}{p+1}\int_{{\mathbb R}^N} |x|^a v_\varepsilon^{p+1}(t_2)dx\bigg) +E(0)
\end{align*}
By monotone convergence, we have 
$$\int_{{\mathbb R}^N}|x|^av_\varepsilon^{p+1}(t_2)dx\to \int_{{\mathbb R}^N}|x|^au^{p+1}(t_2)dx.$$
On the other hand, 
\begin{align*}
 \liminf\limits_{\varepsilon\to 0}\bigg(-\frac{1}{2}\int_{{\mathbb R}^N}|\nabla v_\varepsilon (t_2)|^2dx\bigg)\leq -\liminf\limits_{\varepsilon\to 0}\frac{1}{2}\int_{{\mathbb R}^N}|\nabla v_\varepsilon (t_2)|^2dx \leq -\frac{1}{2}\int_{{\mathbb R}^N}|\nabla v (t_2)|^2dx
\end{align*}
Therefore, (\ref{energyine}) follows. Lemma is proved.
\end{proof}
\begin{remark}{\rm 
 If we assume in addition that $a>-N/2$ then by parabolic regularity, one can see that $u_t\in L^2_{loc}({\mathbb R}^N)$ for any $t\in \mathbb R$. This combined with spatial decay estimates implies 
$$\frac{dE(t)}{dt}= -\|u_t(t)\|^2_{L^2({\mathbb R}^N)},$$
and Lemma~\ref{energy} is then straightforward.}
\end{remark}

\medskip

{\it Proof of Theorem~\ref{th1a}. } We shall prove (i) and (ii) at the same time.  Assume that $u$ is a bounded nontrivial nonnegative solution of (\ref{1}). Then $u$ satisfies spatial decay estimates (\ref{decay1}). Combining with the boundedness of $u$ and $p<p_S(a)$ we have
 \begin{align*}
  |E(u(t))|\leq C, \quad \forall t\in R.
 \end{align*}
It follows from Lemma~\ref{energy} that
\begin{align*}
 \int_{\mathbb R}\int_{{\mathbb R}^N}u_t^2(x,t)dxdt<\infty.
\end{align*}
Consequently, there exists $t_k\to \infty$ such that \begin{align}\label{z1a}
u_t(t_k)\to 0 \quad \text{in}\quad  L^2({\mathbb R}^N).
  \end{align}
We now show that 
\begin{align}\label{z2}
 \|u(t_k)\|_{L^\infty({\mathbb R}^N)}\to 0.
\end{align}
Indeed, if not then there exists $x_k$ such that $u(x_k,t_k)\geq C>0$. It follows from  spatial decay estimates of $u$  that $x_k$ is bounded. We may assume that $x_k\to x_\infty$. Let $v_k(x):=u(x,t_k)$, then there exists a subsequence which converges in $C_{loc}({\mathbb R}^N)$ to a function $v$ satisfying 
$$-\Delta v=|x|^a v^p,$$
and $v(x_\infty)\geq C$. We note also that if $u$ is radial then so is $v$.  This contradicts Liouville-type theorem for Hardy-H\'enon equations (see \cite{ BVG10, PhS}). Hence (\ref{z2}) is true.
 
Let $\alpha_k=\int_{{\mathbb R}^N}|x|^au^{p+1}( t_k)dx$, for any $R>0$ we have 
\begin{align*}
 \alpha_k&=\int_{|x|<R}|x|^au^{p+1}( t_k)dx+\int_{|x|>R}|x|^au^{p+1}( t_k)dx\\
&\leq CR^{N+a}\|u(t_k)\|^{p+1}_{L^\infty({\mathbb R}^N)}+ CR^{-[(2+a)(p+1)/(p-1)-N-a]}.
\end{align*}
Hence, $\limsup_{k\to \infty} \alpha_k \leq  CR^{-[(2+a)(p+1)/(p-1)-N-a]}$. Letting $R\to \infty$, we obtain
\begin{align}\label{z3}
\lim_{k\to \infty} \alpha_k=\lim_{k\to \infty}\int_{{\mathbb R}^N}|x|^au^{p+1}(t_k)dx=0.
\end{align}
We next show that $\|\nabla u(t_k)\|_{L^2({\mathbb R}^N)}\to 0$. Let $\varphi$ be a smooth function in ${\mathbb R}^N$, $0\leq \varphi\leq 1$, $\varphi(x)=0$ in $B_1$, $\varphi(x)=0$ if $|x|\geq 2$ and $|\nabla \varphi|\leq C\varphi^{1/2}$. For any $R>0$, let $\varphi_R(x)=\varphi(x/R)$, then we  have
\begin{align*}
\int_{{\mathbb R}^N}|\nabla u(t_k)|^2\varphi_Rdx&=- \int_{{\mathbb R}^N}u(t_k)\nabla u(t_k).\nabla \varphi_Rdx+\int_{{\mathbb R}^N}\big(|x|^au^{p+1}(t_k)-u_t(t_k)u(t_k)\big)\varphi_Rdx\\
&\leq \frac12 \int_{{\mathbb R}^N}|\nabla u(t_k)|^2\varphi_Rdx+\frac12\int_{{\mathbb R}^N}u^2(t_k)|\nabla\varphi_R|^2\varphi^{-1/2}_Rdx\\
&\quad +\int_{{\mathbb R}^N}\big(|x|^au^{p+1}(t_k)-u_t(t_k)u(t_k)\big)\varphi_Rdx.
 \end{align*}
Using (\ref{z1a})-(\ref{z3}) and the compact support of $\varphi_R$, we deduce that
\begin{align*}
\lim\limits_{k\to \infty} \int_{B_R}|\nabla u(t_k)|^2dx =0,
\end{align*}
for any $R>0$. On the other hand, if follows from spatial decay estimate of $\nabla u$ that 
\begin{align}\label{z3b}
 \int_{{\mathbb R}^N}|\nabla u(t_k)|^2dx \leq \int_{B_R}|\nabla u(t_k)|^2dx+ \int_{|x|>R}|\nabla u(t_k)|^2dx\leq \int_{B_R}|\nabla u(t_k)|^2dx+ CR^{N-2-\frac{4+2a}{p-1}}.
\end{align}
By letting $k\to \infty$ and then $R\to\infty$ in (\ref{z3b}), we obtain
\begin{align}
\lim\limits_{k\to \infty} \int_{{\mathbb R}^N}|\nabla u(t_k)|^2dx=0.
\end{align}
Combinining this with (\ref{z3}) we obtain $E(t_k)\to 0$ as $k\to \infty$. 

Similarly, there exist $s_k\to -\infty$ such that $u_t(s_k)\to 0$ in $L^2({\mathbb R}^N)$ and we deduce that  $E(s_k)\to 0$ as $k\to \infty$.  It follows from Lemma~\ref{energy} that
\begin{align*}
 \int_{s_k}^{t_k}\int_{{\mathbb R}^N}u_t^2(t)dxdt\leq E(s_k)-E(t_k).
\end{align*}
Let $k\to \infty$ we obtain $u_t\equiv 0$. This contradicts Liouville-type theorem for Hardy-H\'enon elliptic equations (see \cite{ BVG10, PhS}).
\qed

\section{Problems with boundary condition}
In this section, we will prove Theorems ~\ref{th5}-\ref{th6}. Let $u$ be a nonnegative solution of problem (\ref{BVP}), and as in the previous section, we denote 
\begin{align}
 E(u(t))=E(t):=\frac12\int_{\Omega}|\nabla u(t)|^2dx-\frac{1}{p+1}\int_{\Omega}|x|^au^{p+1}(t)dx.
\end{align}
Then $E(u(t))$ is well defined and similar to Lemma~\ref{energy}, for $t_1<t_2$, we have 
\begin{align}
 \int_{t_1}^{t_2}\int_{\Omega}u_t^2(t)dxdt\leq-E(t_2)+E(t_1).
\end{align}
We need the following lemma.
\begin{lemma}\label{lem4}
 Consider problem (\ref{BVP}) with nonnegative initial data $u_0\in L^\infty\cap H^1_0(\Omega)$. If $E(u_0)<0$ then $T_{\max}(u_0)<\infty$. 
\end{lemma}

\begin{proof} We follow the concavity method in \cite{Lev73} (see also \cite[Theorem 17.6]{QS07}).  Assume that $T_{\max}(u_0)=\infty$. Let $M(t)=\frac12\int_{0}^{t}\|u(s)\|_2^2ds$ then 
\begin{align*}
 M''(t)&=\int_{\Omega}uu_t(t)dx=-\int_{\Omega}|\nabla u(t)|^2dx+\int_{\Omega}|x|^au^{p+1}(t)dx\\
&=-(p+1)E(u(t))+\frac{p-1}{2}\int_{\Omega}|\nabla u(t)|^2dx\\
&\geq -(p+1)E(u_0)>0.
\end{align*}
Consequently,  $M'(t)\to \infty$ and $M(t)\to\infty$ as $t\to\infty$. Moreover,
\begin{align*}
 M''(t)\geq -(p+1)E(u(t))\geq -(p+1)E(u(t))+(p+1)E(u_0)\geq(p+1)\int_{0}^{t}\|u_t(s)\|_2^2ds,
\end{align*}
hence
\begin{align*}
 M(t)M''(t)&\geq \frac{p+1}{2}\left(\int_{0}^{t}\|u_t(s)\|_2^2ds\right) \left(\int_{0}^{t}\|u(s)\|_2^2 ds\right)\\
&\geq \frac{p+1}{2}\left(\int_{0}^{t}\int_{\Omega}u(x,s)u_t(x,s)dxds\right)^2\\
&=\frac{p+1}{2}\left(M'(t)-M'(0)\right)^2.
\end{align*}
Since $M'(t)\to \infty$ as $t\to \infty$, there exist $\alpha, t_0>0$ such that
$$M(t)M''(t)\geq (1+\alpha)(M'(t))^2, \quad t\geq t_0.$$
This guarantees that the nonincreasing function $t\mapsto M^{-\alpha}(t)$ is concave on $[t_0,\infty)$ which contradicts the fact that $M^{-\alpha}(t)\to 0$ as $t\to\infty$.
\end{proof}

{\it Proof of Theorem~\ref{th5}. } Assume that the bound (\ref{z4}) does not hold for global nonnegative solutions. Then there exist $t_k>0$ and $u_{0,k}\geq 0$ such that $\|u_{0,k}\|_\infty\leq C_0$ and the solutions $u_k$ with initial data $u_{0,k}$ satisfying
\begin{align}
 M_k:=u_k(x_k,t_k)=\sup\{u_k(x,t): x\in\Omega,\; t\in [0, t_k]\} \to \infty \quad\text { as } \quad k\to \infty.
\end{align}
By the point fixed argument, there exitss $\delta>0$ such that 
$$u_k(x,t)\leq 2C_0, \quad \forall (x,t)\in \Omega\times [0,\delta].$$
Hence $t_k\geq \delta$ for $k$ large enough. We will show by variation-of-constants formula that 
\begin{align}
\label{w1}
\sup_k\|u_k(\delta/2)\|_{H^1(\Omega)}\leq C.
 \end{align}
Indeed, (\ref{w1}) is straightforward if $a\geq 0$. If $a<0$  then 
\begin{align*}
\|\nabla u_k(\delta/2\|_2\leq C\|u_{0,k}\|_2+ C\int_{0}^{\delta/2}s^{-\frac12-\frac{N}{2}(\frac{1}{q}-\frac{1}{2})}\| |x|^a\|_q\, ds.
 \end{align*}
We now choose $q=1$ when $N=1$, and $q=\frac{N}{\gamma+|a|}$ when $N\geq 2$, where $\gamma>0$ is small such that $\gamma+|a|<2$. Then  $|x|^a\in L^q(\Omega)$ and $\frac12+\frac{N}{2}(\frac{1}{q}-\frac{1}{2})<1$. Hence (\ref{w1}).

Combining (\ref{w1}) with Lemma~\ref{lem4}, we obtain
\begin{align}\label{z6}
 0\leq E(u_k(\delta/2)\leq C.
\end{align}
We may assume that $x_k\to x_0\in \bar{\Omega}$. We denote $d_k=\text{dist}(x_k, \partial \Omega)$.  

{\it A. Nonradial case.}

{\it Case $A_1$: $x_0\in\Omega\setminus\{0\}$. }   Let
\begin{align}\label{w2}
 v_k(y,s)=\lambda_k^{2/(p-1)}u_k(x_k+\lambda_ky, t_k+\lambda_k^2s), \quad (y,s)\in {Q}_k,
\end{align}
where $\lambda_k=M_k^{-(p-1)/2}$ and  $Q_k=\{(y,s): (x_k+\lambda_ky, t_k+\lambda_k^2s)\in \Omega\times (0, t_k)\}$. It follows that $0\leq v_k\leq 1=v_k(0,0)$ and $v_k$ solves the problem
\begin{align*}
 \partial_sv_k-\Delta v_k&=|x_k+\lambda_k y|^av_k^p, \qquad (y,s)\in \tilde{Q}_k:=\{(y,s): |y|<\frac{d_k}{\lambda_k}, -\frac{t_k}{2\lambda_k^2}<s<0\}.
\end{align*}

Using parabolic $L^ p$-estimates together with standard embedding theorems, we may assume $v_k\to v$ in $C_{loc}^{0,0}({\mathbb R}^N\times (-\infty, 0))$, and 
$$v_s-\Delta v= |x_0|^a v^p, \quad \text{in}\; {\mathbb R}^N\times (-\infty, 0).$$
with $0\leq v\leq v(0,0)=1$. Using (\ref{z6}), we have
\begin{align*}
 \int_{\tilde{Q}_k}|\partial_sv_k|^2dyds&=\lambda_k^{\frac{4}{p-1}-N+2}\int_{t_k/2}^{t_k}\int_{|x-x_k|<d_k}|\partial_t u_k|^2dxdt\leq \lambda_k^{\frac{4}{p-1}-N+2}\int_{\delta/2}^{\infty}\int_{\Omega}|\partial_t u_k|^2dxdt\\
&\leq \lambda_k^{\frac{4}{p-1}-N+2}\left[E(u_k(\delta/2))-\lim_{t\to\infty}E(u_k(t))\right] \to 0 \text{ as } k\to \infty.
\end{align*}
Since $\partial_sv_k\to v_s$ in $ {\mathcal D'}({\mathbb R}^N\times (-\infty, 0)$ then $v_s\equiv 0$. This contradicts the Liouville-type theorem for Lane-Emden equation (see \cite{GS81a}).

{\it Case $A_2$: $x_0\in \partial\Omega$. } We rescaling $u$ as in (\ref{w2}).

If $d_k/\lambda_k\to \infty$ then we have the same contradition as in the first case. If $d_k/\lambda_k\to c>0$ then similarly, we have a function $v$ solving the problem
\begin{align*}
\begin{cases}
  v_s-\Delta v=l v^p\quad &\text{in } \; H_c^N\times(-\infty,0),\\
v=0&\text{on } \;  \partial H_c^N\times(-\infty,0),
\end{cases}
\end{align*}
and satisfying $0\leq v\leq v(0,0)=1$, where $ H_c^N:=\{y\in {\mathbb R}^n: y_1>-c\}$. As in case $A_1$, we have $v_s=0$, hence contradition.

{ \it Case $A_3$: $x_0=0$. } We have the following two possibilities:

(i) If $M_k|x_k|^{(2+a)/(p-1)}\leq C$, then let $\lambda_k=M_k^{-(p-1)/(2+a)}$ we have  $\lambda_k^{-1}x_k$ is bounded. We may assume that  $\lambda_k^{-1}x_k\to P$. Let 
\begin{align*}
 w_k(y,s)=\lambda_k^{(2+a)/(p-1)}u_k(x_k+\lambda_ky, t_k+\lambda_k^2s),
\end{align*}
Then $w_k$ solves
\begin{align*}
 \partial_sw_k-\Delta w_k&=|\lambda_k^{-1}x_k+ y|^aw_k^p, \quad (y,s)\in \tilde{Q}_k:=\{(y,s): |y|<\frac{d_k}{\lambda_k}, -\frac{t_k}{2\lambda_k^2}<s<0\}. 
\end{align*}
A similar limiting procedure as in Case $A_1$ then produces a solution $w$ of 
\begin{align*}
 w_s-\Delta w=|y+P|^a w^p, \quad (y,s)\in {\mathbb R}^N\times (-\infty,0).
\end{align*}
with $0\leq w\leq w(0,0)=1$. Using (\ref{z6}) and $p<p_S(a)$, we have
\begin{align*}
 \int_{\tilde{Q}_k}|\partial_sw_k|^2dyds&=\lambda_k^{(4+2a)/(p-1)-N+2}\int_{t_k/2}^{t_k}\int_{|x-x_k|<d_k}|\partial_t u_k|^2dxdt\\
&\leq \lambda_k^{(4+2a)/(p-1)-N+2}\int_{\delta/2}^{\infty}\int_{\Omega}|\partial_t u_k|^2dxdt\\
&\leq \lambda_k^{(4+2a)/(p-1)-N+2}\left[E(u_k(\delta/2))-\lim_{t\to\infty}E(u_k(t))\right] \to 0 \text{ as } k\to \infty.
\end{align*}

Since $\partial_sw_k\to w_s$ in $ {\mathcal D'}({\mathbb R}^N\times (-\infty, 0)$ then $w_s\equiv 0$. Hence $-\Delta w=|y+P|^a w^p$ in ${\mathbb R}^N$ with $w(0,0)=1$. 
 After a spatial shift, we have a contradiction with Liouville-type theorem for Hardy-H\'enon elliptic equation (see \cite{PhS, BVG10}).

(ii) If there exists a subsequence of $k$, still denoted by $k$, such that $M_k|x_k|^{(2+a)/(p-1)}\to \infty$, then we can choose $m_k>1$ such that
$$M_k|x_k|^{(2m_k+a)/(p-1)}=1.$$
Let
\begin{align*}
 w_k(y,s)=\lambda_k^{\frac{2m_k+a}{m_k(p-1)}}u_k(x_k+\lambda_ky, t_k+\lambda_k^2s), \quad (y,s)\in {Q}_k,
\end{align*}
where $\lambda_k=M_k^{-m_k(p-1)/(2m_k+a)}$, then $w_k$ solves the problem
\begin{align*}
 \partial_s w_k-\Delta w_k=|\lambda_k^{-1/m_k}x_k+\lambda_k^{1-1/m_k}y |^a w_k^p, \quad (y,s)\in \tilde{Q}_k:=\{(y,s): |y|<\frac{d_k}{\lambda_k}, -\frac{t_k}{2\lambda_k^2}<s<0\}.
\end{align*}
 Since $\lambda_k^{-1/m_k}|x_k|=1$ and $0<\lambda_k^{1-1/m_k}<1$ , we may assume that $\lambda_k^{-1/m_k}x_k\to P$ with $|P|=1$ and $\lambda_k^{1-1/m_k}\to l\in [0,1]$. A similar limiting procedure as in Case $A_1$  then produces a solution $w$ of 
\begin{align*}
 w_s-\Delta w=|P+ly|^a w^p, \quad (y,s)\in {\mathbb R}^N\times (-\infty,0).
\end{align*}
We will show that $w_s\equiv 0$, indeed,
\begin{align*}
 \int_{\tilde{Q}_k}|\partial_sw_k|^2dyds&=\lambda_k^{(4m_k+2a)/m_k(p-1)-N+2}\int_{t_k/2}^{t_k}\int_{|x-x_k|<d_k}|\partial_t u_k|^2dxdt\\
&\leq \lambda_k^{(4m_k+2a)/m_k(p-1)-N+2}\int_{\delta/2}^{\infty}\int_{\Omega}|\partial_t u_k|^2dxdt\\
&\leq \lambda_k^{(4m_k+2a)/m_k(p-1)-N+2}\left[E(u_k(\delta/2))-\lim_{t\to\infty}E(u_k(t))\right] \to 0.
\end{align*}
since  
\begin{align*}
 (4m_k+2a)/m_k(p-1)-N+2\geq \min \{4/(p-1)-N+2, (4+2a)/(p-1)-N+2\}>0.
\end{align*}
Therefore $w_s\equiv 0$, and we have the contradiction. 
\medskip

{\it B. Radial case. } Assume $\Omega=B_R$, we will write $u_k=u_k(r,t)$, $r\in (0,R)$, $M_k=M_k(r,t)$, where $r= |x|$. Then $u_k$ solves the equation
\begin{align*}
u_t-u_{rr}-\frac{N-1}{r}u_r=r^au^p.
\end{align*}
Let $r_k=|x_k|$, we have 3 subcases:

 {\it Case $B_1$: $r_k\to r_0\in (0,R)$. } 
Let $\lambda_k= M_k^{-(p-1)/2}(r_k, t_k)$, we rescale by
\begin{align*}
 v_k(\rho,s):=\lambda_k^{2/(p-1)}u_k(r_k+\lambda_k\rho, t_k+\lambda_k^2 s), \quad (\rho, s)\in \big(0,\; \frac{R-r_k}{\lambda_k}\big)\times(-t_k/\lambda_k^2, 0)
\end{align*}
Then $v_k$ solves the equation
\begin{align*}
 v_s-v_{\rho\rho}-\frac{N-1}{\rho+r_k/\lambda_k}v_\rho=|r_k+\lambda_k\rho|^av^p, 
\end{align*}
we note that $v_k(0,0)=1$, after extracting a subsequence, we can assume that $v_k\to v$ that satisfies
$$v_s-v_{\rho\rho}=r_0^av^p \quad \text{ in } \quad {\mathbb R}\times (-\infty, 0)$$
and
$$v(0,0)=1.$$
By the argument similar to the case $A_1$, we have $v_s=0$ and a contradiction with Liouville-type theorem for Lane-Emden equation with $N=1$.

{\it Case $B_2$: $r_k\to R$.} We have the following two possibilities:

(i) If $\frac{R-r_k}{\lambda_k}\to \infty$. The same rescaling as in case $B_1$ leads to a contradiction as in  case $B_1$.

(ii) If $\frac{R-r_k}{\lambda_k}\to c$. The same rescaling as in case $B_1$ leads to a contradiction with the Liouville-type theorem in half space for Lane-Emden equation with $N=1$.

{\it Case $B_3$: $r_k\to 0$.  } It follows from  the singularity estimate in Theorem~\ref{th2}(i) and $t_k\geq \delta$ that 
 \begin{align*}
  M_kr_k^{(2+a)/(p-1)}\leq C.
 \end{align*}
Let $\lambda_k=M_k^{-(p-1)/(2+a)}$ then $\lambda_k^{-1}r_k$ is bounded. We may assume that $\lambda_k^{-1}r_k\to P$. Let
\begin{align*}
 w_k(\rho,s)=\lambda_k^{(2+a)/(p-1)}u_k(\lambda_k\rho, t_k+\lambda_k^2s),\quad 0<\rho<R/\lambda_k, \quad -t_k/\lambda_k^2<s<0.
\end{align*}
Then $w_k$ solves
\begin{align*}
 \partial_sw-w_{\rho\rho}-\frac{N-1}{\rho}w_\rho=|\rho|^aw^p.
\end{align*}
After extracting a subsequence, we can assume that $w_k\to w$ in $C^{0,0}_{loc}({\mathbb  R}\times (-\infty, 0))$ and
\begin{align*}
 \partial_sw-w_{\rho\rho}-\frac{N-1}{\rho}w_\rho=| \rho|^aw^p.
\end{align*}
By similar argument as in case $A_3$, we have $w_s=0$. We therefore have a contradiction with Liouville-type theorem for radial solutions of Hardy-H\'enon elliptic equation.  Theorem is proved.
\qed

\medskip
We now turn to prove Theorem~\ref{th4}.

{\it Proof of Theorem~\ref{th4}.} It suffices to prove assertion (i).
 
Suppose that estimate (\ref{universalbound}) is false. Then there exist sequences $T_k\in(0, \infty)$, $u_k$, $y_k\in \Omega$, $s_k\in (0, T_k)$, such that $u_k$ solves problem (\ref{BVP}) (with $T$ replaced by $T_k$) and the functions 
\begin{align*}
M_k:=u_k^{(p-1)/2}
\end{align*}
satisfy
\begin{align*}
M_k(y_k, s_k)>2k(1+d_k^{-1}(s_k)),
\end{align*}
where $d_k:=(\min(t, T_k-t))^{1/2}$.
We apply Doubling Lemma in \cite[Lemma 5.1]{PQS07}, with $X={\mathbb R}^{N+1}$, equipped with parabolic distance (\ref{dist}), $\Sigma=\Sigma_k=\bar{\Omega}\times [0,T_k]$, $D=D_k=\bar{\Omega}\times (0,T_k)$, and $\Gamma=\Gamma_k=\bar{\Omega}\times \{0,T_k\}$. Notice that
$$d_k(t)=d_P((x,t), \Gamma_k), \quad (x,t)\in \Sigma_k.$$
Then there exist $x_k\in\Omega$, $t_k\in(0,T_k)$ such that
\begin{align*}
&M_k(x_k, t_k)> 2k d_k^{-1}(t_k),\\
&M_k(x_k, t_k)\geq M_k(y_k, s_k)>2k,
\end{align*}
and 
\begin{align}\label{z7}
M_k(x,t)\leq 2 M_k(x_k, t_k), \quad (x,t)\in D_k\cap \tilde{B}_k,
\end{align}
where
\begin{align*}
\tilde{B}_k=\{(x,t)\in {\mathbb R}^{N+1}: |x-x_k|+|t-t_k|^{1/2}\leq kM_k^{-1}\}.
\end{align*}
We may assume that $x_k\to x_0\in\bar{\Omega}$. Let $d_k=\text{dist}(x_k,\partial \Omega)$. 

{\it A. The nonradial case.}  We have 3 subcases.

{\it Case $A_1$: $x_0\in\Omega\setminus\{0\}$. }   Let
\begin{align*}
 v_k(y,s)=\lambda_k^{2/(p-1)}u_k(x_k+\lambda_ky, t_k+\lambda_k^2s), \quad (y,s)\in \tilde{D}_k,
\end{align*}
where $\lambda_k=M_k^{-1}$ and 
 $\tilde{D}_k=\left(\lambda_k^{-1}(\Omega-x_k)\cap \{|y|<\frac{k}{2}\}\right)\times \left(-\frac{k^2}{4},\frac{k^2}{4} \right)$.

 We have $v_k(0,0)=1$, and it follows from (\ref{z7}) that $v_k\leq 2^{2/(p-1)}$,  and $v_k$ solves the problem
\begin{align*}
 \partial_sv_k-\Delta v_k&=|x_k+\lambda_k y|^av_k^p, \qquad (y,s)\in \tilde{D}_k.
\end{align*}

Using parabolic $L^ p$-estimates together with standard embedding theorems, we may assume $v_k\to v$ in $C_{loc}^{0,0}({\mathbb R}^N\times {\mathbb R})$, and 
$$v_s-\Delta v= |x_0|^a v^p, \quad \text{in}\; {\mathbb R}^N\times {\mathbb R}.$$
with $v(0,0)=1$. This contradicts Theorem~A(ii).

{\it Case $A_2$: $x_0\in \partial\Omega$. } We rescale $u$ as in case $A_1$. Let
\begin{align*}
 v_k(y,s)=\lambda_k^{2/(p-1)}u_k(x_k+\lambda_ky, t_k+\lambda_k^2s),
\end{align*}
If $d_k/\lambda_k\to \infty$ then we have the same contradition as in  case $A_1$. If $d_k/\lambda_k\to c>0$ then similarly, we have a function $v$ solving the problem
\begin{align*}
\begin{cases}
  v_s-\Delta v=l v^p\quad &\text{in } \; H_c^N\times{\mathbb R},\\
v=0&\text{on } \;  \partial H_c^N\times{\mathbb R},
\end{cases}
\end{align*}
and satisfying $v(0,0)=1$, where $ H_c^N:=\{y\in {\mathbb R}^n: y_1>-c\}$. This contradicts \cite[Theorem 2.19(ii)]{PQS07b}.

{ \it Case $A_3$: $x_0=0$. } We have two possibilities:

(i) If $M_k^{2/(2+a)}|x_k|\leq C$, then letting $\lambda_k=M_k^{-2/(2+a)}$, we have  $\lambda_k^{-1}x_k$ is bounded. We may assume that  $\lambda_k^{-1}x_k\to P$. Let 
\begin{align*}
 w_k(y,s)=\lambda_k^{(2+a)/(p-1)}u_k(x_k+\lambda_ky, t_k+\lambda_k^2s),
\end{align*}
Then $w_k$ solves
\begin{align*}
 \partial_sw_k-\Delta w_k&=|\lambda_k^{-1}x_k+ y|^aw_k^p, \quad (y,s)\in \tilde{D}_k.
\end{align*}
A similar limiting procedure as in Case $A_1$ then produces a solution $w$ of 
\begin{align*}
 w_s-\Delta w=|y+P|^a w^p, \quad (y,s)\in {\mathbb R}^N\times {\mathbb R}.
\end{align*}
 After a spatial shift, we have a contradiction with Theorem~\ref{th1a}(i).

(ii) If there exists a subsequence of $k$, still denoted by $k$, such that $M_k^{2/(2+a)}|x_k|\to \infty$. We can choose $m_k>1$ such that
$$M_k^{2/(2m_k+a)}|x_k|=1.$$
Let
\begin{align*}
 w_k(y,s)=\lambda_k^{\frac{2m_k+a}{m_k(p-1)}}u_k(x_k+\lambda_ky, t_k+\lambda_k^2s), \quad (y,s)\in \tilde{D}_k,
\end{align*}
where $\lambda_k=M_k^{-2m_k/(2m_k+a)}$, then  $w_k$ solves the problem
\begin{align*}
 \partial_s w_k-\Delta w_k=|\lambda_k^{-1/m_k}x_k+\lambda_k^{1-1/m_k} |^a w_k^p, \quad (y,s)\in \tilde{D}_k.
\end{align*}
 Since $\lambda_k^{-1/m_k}|x_k|=1$ and $0<\lambda_k^{1-1/m_k}<1$ then we may assume that $\lambda_k^{-1/m_k}x_k\to P$ with $|P|=1$ and $\lambda_k^{1-1/m_k}\to l\in [0,1]$. A similar limiting procedure as in case $A_1$ then produces a solution $w$ of 
\begin{align*}
 w_s-\Delta w=|P+ly|^a w^p, \quad (y,s)\in {\mathbb R}^N\times {\mathbb R}.
\end{align*}
We have a contradition with Theorem~\ref{th1a}(i) if $l\ne 0$, or with Theorem A(ii) if $l=0$.

{\it B. The radial case. }
Assume $\Omega=B_R$, we will write $u_k=u_k(r,t)$, $r\in (0,R)$, $M_k=M_k(r,t)$, where $r=|x|$. Then $u_k$ solves the equation
\begin{align*}
u_t-u_{rr}-\frac{N-1}{r}u_r=r^au^p.
\end{align*}
Denote $r_k=|x_k|$, we have 3 cases:

 {\it Case $B_1$: $r_k\to r_0\in (0,R)$. } 
Let $\lambda_k= M_k^{-1}(r_k, t_k)$, we rescale by
\begin{align*}
 v_k(\rho,s):=\lambda_k^{2/(p-1)}u_k(r_k+\lambda_k\rho, t_k+\lambda_k^2 s), \quad \rho<\min\{\frac{k}{2}, \frac{r_0}{\lambda_k}, \frac{R-r_0}{\lambda_k}\},\; |s|< \frac{k^2}{4}.
\end{align*}
Then $v_k$ solves the equation
\begin{align*}
 v_s-v_{\rho\rho}-\frac{N-1}{\rho+r_k/\lambda_k}v_\rho=|r_k+\lambda_k\rho|^av^p, 
\end{align*}
we note that $v_k(0,0)=1$, after extracting a subsequence, we can assume that $v_k\to v$ that satisfies
$$v_s-v_{\rho\rho}=r_0^av^p \quad \text{ in } \quad {\mathbb R}\times {\mathbb R}$$
and
$$v(0,0)=1.$$
This contradicts Theorem A(ii) for $N=1$.

{\it Case $B_2$: $r_k\to R$.} We have the following two possibilities:

(i)  $\frac{R-r_k}{\lambda_k}\to \infty$. The same rescaling as in case $B_1$ lead to a contradiction with Theorem A(ii) for $N=1$.

(ii) $\frac{R-r_k}{\lambda_k}\to c$ The same rescaling as in case $B_1$ lead to a contradiction with the Liouville-type theorem in half space with $N=1$  (see \cite[Theorem 2.19]{PQS07b}).

{\it Case $B_3$: $r_k\to 0$.  } We have the following two possibilities:

(i)  $M^{2/(2+a)}_k(r_k, t_k)r_k\leq C$.  
Let $\lambda_k=M_k^{-2/(2+a)}(r_k, t_k)$ then $\lambda_k^{-1}r_k$ is bounded. We may assume that $\lambda_k^{-1}r_k\to P$. Let
\begin{align*}
 w_k(\rho,s)=\lambda_k^{(2+a)/(p-1)}u_k(r_k+\lambda_k\rho, t_k+\lambda_k^2s), \quad \rho<\min\{\frac{k}{2}, \frac{R}{2\lambda_k}\},\; |s|< \frac{k^2}{4}.
\end{align*}
Then $w_k$ solves
\begin{align*}
 \partial_sw-w_{\rho\rho}-\frac{N-1}{\rho}w_\rho=|\rho|^aw^p.
\end{align*}
After extracting a subsequence, we can assume that $w_k\to w$ in $C^{0,0}_{loc}({\mathbb R}\times {\mathbb R})$ and
\begin{align*}
 \partial_sw-w_{\rho\rho}-\frac{N-1}{\rho}w_\rho=| \rho|^aw^p.
\end{align*}
So we have a contradiction with Theorem \ref{th1a}(ii)

(ii) There exists a subsequence of $k$, still denoted by $k$ such that $M^{2/(2+a)}_k(r_k, t_k)r_k \to \infty$. We can choose $m_k>1$ such that
$$M_k^{2/(2m_k+a)}(r_k, t_k)r_k=1.$$
Let $\lambda_k=M_k^{-2m_k/(2m_k+a)}(r_k, t_k)$ then $\lambda_k^{-1/m_k}r_k=1$ and $0<\lambda_k^{1-1/m_k}<1$, we may assume that $\lambda_k^{1-1/m_k}\to l\in [0,1]$.

If $l=0$ then we rescale
\begin{align*}
 w_k(\rho,s)=\lambda_k^{\frac{2m_k+a}{m_k(p-1)}}u_k(r_k+\lambda_k\rho, t_k+\lambda_k^2s), \quad \rho<\min\{\frac{k}{2}, \frac{R}{2\lambda_k}\},\; |s|< \frac{k^2}{4}.
\end{align*}
It follows that $w_k$ solves the problem
\begin{align*}
 \partial_s w-w_{\rho\rho}-\frac{N-1}{\rho+\lambda_k^{-1}r_k}w_\rho=|\lambda_k^{-1/m_k}r_k+\lambda_k^{1-1/m_k} \rho|^a w_k^p.
\end{align*}
 A similar limiting procedure as in Case $A_1$ then produces a solution $w$ of 
\begin{align*}
 w_s-w_{\rho\rho}= w^p, \quad (\rho,s)\in {\mathbb R}\times {\mathbb R}, 
\end{align*}
with $w(0,0)=1$, and we have a contradiction with Theorem A(ii) for $N=1$.

If $l\ne 0$ then we rescale
\begin{align*}
 w_k(\rho,s)=\lambda_k^{\frac{2m_k+a}{m_k(p-1)}}u_k(\lambda_k\rho, t_k+\lambda_k^2s), \quad \rho<\min\{\frac{k}{2}, \frac{R}{\lambda_k}\},\; |s|< \frac{k^2}{4},
\end{align*}
It follows that $w_k$ solves the problem
\begin{align*}
 \partial_s w-w_{\rho\rho}-\frac{N-1}{\rho}{w}_\rho=|\lambda_k^{1-1/m_k} \rho|^a w^p.
\end{align*}
Passing to the limit, we obtain $w$ solutions to 
\begin{align*}
 w_s-w_{\rho\rho}-\frac{N-1}{\rho}w_\rho=l^a|\rho|^a w^p, \quad (\rho,s)\in {\mathbb R}\times {\mathbb R}.
\end{align*}
with $w(0,0)=1$. This contradicts Theorem~\ref{th1a}(i). 
\qed

\medskip
We now give proof of Theorem \ref{th6}. We first need the following lemma, which is proved by the same argument as in \cite[Lemma 4.1]{QSW04}.
\begin{lemma}\label{lem5}
 Let $a>0$, $p>1+\frac{a}{N}$ and $\varepsilon\in (0,(p+1)/2)$. Then there exists $a_{\varepsilon}\in (0, a)$ such that, for any nonnegative solution  $u$   of (\ref{BVP}) and $t\in (0,T/2)$, it holds
\begin{align*}
\int_{0}^{t}\int_{\Omega}|x|^{a_{\varepsilon}}u^{\frac{p+1}{2}-\varepsilon}dxdt\leq C(p,a,\Omega,\varepsilon)(1+t)\left(1+T^{-1/(p-1)}\right).
\end{align*}
\end{lemma}
\begin{proof}
The proof is nearly the same as in  \cite[Lemma 4.1]{QSW04}. Only one thing we take into consideration is the condition $p>1+\frac{a}{N}$, which implies the H\"older 's inequality
$$\int_{\Omega}|x|^au^p(t)\varphi_1dx\geq C\bigg(\int_{\Omega}u(t)\varphi_1dx\bigg)^p.$$
\end{proof}

{\it Proof of Theorem \ref{th6}. }
If $p<p_B$ then the estimate (\ref{universal}) is a consequence of Theorem~\ref{th4}(ii). We may assume that $p_B\leq p< \frac{N+2+a}{N-2+a}$ (this in particular implies $a<\frac{N+2}{4N-1}$ and $N\geq 2$). We shall follow the steps similar to those in \cite{QSW04}. In order not to repeat the same things, we only precise the modifications and the differences coming from the weight term $|x|^a$.

By Theorem \ref{th5}, we know that global solutions of problem (\ref{BVP}) satisfy the a priori estimate
\begin{align*}
 \|u(t)\|_\infty \leq C(\Omega, p, a, \|u(t_0)\|_\infty),\quad  t\geq t_0\geq 0,
\end{align*}
where $C$ remains bounded for $\|u(t_0)\|_\infty$ bounded. Therefore, it is sufficient to show the existence of $C(\Omega, p, a, \tau)>0$ such that any global solution $u$ of problem (\ref{BVP}) satisfies
\begin{align}
 \inf_{t\in (0,\tau)}\|u(t)\|_\infty \leq C(\Omega, p, a, \tau).
\end{align}
Using the boundedness of $|x|^a$ in $\Omega$, by the well-known estimate (see \cite{Wei80}), we have
\begin{align}\label{az1}
 \|u(t)\|_\infty\leq C\|u(t_0)\|_q (t-t_0)^{-N/(2q)},\quad 0\leq t-t_0\leq T(\|u(t_0)\|_q)
\end{align}
where $r\in [q, \infty]$ and $q>q_c:=N(p-1)/2$.
We note that $p\geq p_B>1+\frac{2a}{N}$, hence $q_c>1$.
Since $p<\frac{N+2+a}{N-2+a}$, we can choose $q\in (q_c, p+1)$ such that $aq/(p+1-q)<N$. It follows from (\ref{az1}) and H\"older inequality that 
\begin{align*}
 \|u(t)\|_\infty\leq C\|u(t_0)\|_q (t-t_0)^{-N/(2q)}\leq C\| |x|^{a/(p+1)}u(t_0)\|_{p+1} (t-t_0)^{-N/(2q)}.
\end{align*}
Therefore, (\ref{universal}) will follow if we can show that
\begin{align*}
 \inf_{t\in (0,\tau)} \||x|^{a/(p+1)}u(t)\|_{p+1}\leq C(\Omega, p,a,\tau).
\end{align*}
We argue by contradition, assume that for each $k$, there exists a global solution $u_k\geq 0$ such that 
\begin{align}\label{z10}
 \int_{\Omega}|x|^au^{p+1}_k(t)dx >k, \quad \forall t\in (0,\tau/2).
\end{align}
Denote 
\begin{align*}
E_k(t)=E(u_k(t))=\frac{1}{2}\int_{\Omega}|\nabla u_k(t)|^2dx-\frac{1}{p+1}\int_{\Omega}|x|^au_k^{p+1}(t)dx.
\end{align*}
Then $E_k'(t)=-\|\partial_t u_k(t)\|_2^2\leq 0$ and $u_k$ satisfies the identity
\begin{align}\label{z12a}
\frac{1}{2}\frac{d}{dt}\int_{\Omega}u_k^2(t)dx&=-\int_{\Omega}\|\nabla u_k\|^2(t)dx+\int_{\Omega}|x|^au_k^{p+1}(t)dx \\
&=-2E_k(t)+\frac{p-1}{p+1}\int_{\Omega}|x|^au_k^{p+1}(t)dx.\label{z12}
\end{align}
{\it Step 1.} We claim that 
\begin{align}\label{z11}
 E_k(\tau/4)\geq k^{1/2}
\end{align}
for all $k\geq k_0(\Omega,p,a)$ large enough.

Assume that (\ref{z11}) is false. Since $p>1+\frac{2a}{N}$, we have
\begin{align}\label{bs1}
 \|u_k\|_{2}\leq C \| |x|^{a/p+1}u_k\|_{p+1}.
\end{align}
 Using (\ref{z12}) and (\ref{bs1}),  for all $t\geq  \tau/4$, we have
\begin{align}\label{z13}
\frac{1}{2}\frac{d}{dt}\int_{\Omega}u_k^2(t)dx\geq -2k^{1/2}+\frac{p-1}{p+1}\int_{\Omega}|x|^au_k^{p+1}(t)dx\geq -2k^{1/2}+C\left(\int_{\Omega}u_k^2(t)dx\right)^{(p+1)/2}.
\end{align}
This implies 
\begin{align}\label{z14}
\int_{\Omega}u_k^2(t)dx\leq Ck^{\frac{1}{p+1}}, \quad t\geq \tau/4.
\end{align}
Combining (\ref{z10}) with (\ref{z13}), we deduce 
\begin{align*}
\frac{1}{4}k\tau\leq \int_{\tau/4}^{\tau/2}\int_{\Omega}|x|^au_k^{p+1}(t)dxdt\leq C(k^{\frac{1}{p+1}}+k^{1/2}\tau),
\end{align*}
which gives a contradiction for $k$ large enough.

{\it Step 2.} Let $\alpha >0$ to be fixed later and $F_k=\{t\in(0,\tau/4] : -E_k'(t)\geq E_k^{1+1/\alpha}(t)\}$. By the same argument as in \cite{QSW04} we have,  $|F_k|<\tau/8$ for all $k\geq k_0$ large enough.

{\it Step 3.} Choose $\alpha \geq (p+1)/(p-1)$. We claim that for all $k\geq k_0$ large,
\begin{align}\label{z15}
\|\partial_tu_k(t)\|_2^2\leq C\left(\int_{\Omega}|x|^au_k^{p+1}(t)dx\right)^{(\alpha+1)/ \alpha}, \quad \text{ for } t\in (0,\tau/4]\setminus F_k.
\end{align}
For all $t\in (0,\tau/4]\setminus F_k$, we have
\begin{align}\label{z15b}
 \|\partial_tu_k(t)\|_2^2=-E'_k\leq E_k^{1+1/\alpha}(t)\leq \|\nabla u_k(t)\|_2^{2(1+1/\alpha)}.
\end{align}
This along with (\ref{z12a}) and (\ref{bs1}) implies
\begin{align*}
 \|\nabla u_k(t)\|_2^2&\leq \int_{\Omega}|x|^au_k^{p+1}(t)dx + \|u_k(t)\|_2\|\partial_t u_k(t)\|_2\\
&\leq \int_{\Omega}|x|^au_k^{p+1}(t)dx + \| |x|^{a/p+1}u_k(t)\|_{p+1}\|\nabla u_k(t)\|_2^{1+1/\alpha}\\
&\leq \int_{\Omega}|x|^au_k^{p+1}(t)dx+ C\| |x|^{a/p+1}u_k(t)\|_{p+1}^{2\alpha/(\alpha-1}+\frac12 \|\nabla u_k(t)\|_2^2\\
&\leq C\int_{\Omega}|x|^au_k^{p+1}(t)dx+ \frac12 \|\nabla u_k(t)\|_2^2
\end{align*}
where we have used $\alpha\geq (p+1)/(p-1)$ and (\ref{z10}). Consequently,
$$\|\nabla u_k(t)\|_2^2\leq C\int_{\Omega}|x|^au_k^{p+1}(t)dx.$$
Combing this with (\ref{z15b}), we have (\ref{z15}).

{\it Step 4.} Let $0<q<(p+1)/2$, $b=(p+1-q)(\alpha+1)/\alpha$ and
$$G_k=\{t\in (0,\tau/4] : \|\partial_t u_k(t)\|_2^2\leq C\|u_k\|_\infty^b\}.$$
We claim that $|G_k|>0$.

Due to $a>a_{\varepsilon}$ and Lemma~\ref{lem5}, for $A=A(p,q,a,\Omega, \tau)$ large enough, the set
\begin{align*}
 \tilde{G}_k:=\{t\in (0,\tau/4]: \int_{\Omega}|x|^{a}u_k^q\geq A\}
\end{align*}
satisfies
\begin{align}\label{z20}
 |\tilde{G}_k|<\tau/8.
\end{align}
It follows from (\ref{z10}) that, for all $t\in(0,\tau/4]\setminus \tilde{G}_k$,
\begin{align*}
 \int_{\Omega}|x|^au_k^{p+1}(t)dx\leq C\|u_k(t)\|_\infty^{p+1-q}\int_{\Omega}|x|^au_k^{q}(t)dx\leq C\|u_k(t)\|_\infty^{p+1-q}. 
\end{align*}
Therefore, $G_k\supset (0,\tau/4]\setminus (F_k\cup\tilde{G}_k)$. The claim follows from Step 2 and (\ref{z20}).

{\it Step 5.} Construction of a sequence of rescaling times.

If $N\leq 3$, for each $k$ large, we just pick any $t_k\in G_k$.

If $N>3$, we follows the argument in Step 5 of \cite{QSW04}, (since $|x|^a$ is bounded in $\Omega$, Lemma~5.2 and 5.3 in \cite{QSW04} still hold),  and there exists $t_k\in (0, \tau)$ such that 
\begin{align}\label{z16}
\|\partial_t u_k(t_k)\|_r\leq \|u_k(t_k)\|_\infty^{\frac{b}{2}+\beta(p-1)}, \quad 2\leq r\leq \infty.
\end{align}
where $\beta=\frac{N}{2}(\frac12-\frac1r)$. (Note that when $N\leq 3$, (\ref{z16}) holds for  $r=2$, $\beta=0$.)

{\it Step 6.} We will now obtain a contradiction by using rescaling argument. By (\ref{z10}), we have $M_k=\|u_k(t_k)\|_\infty \to \infty$. Let $x_k\in \bar{\Omega}$ be such that $M_k=u_k(x_k, t_k)$. We may assume that $x_k\to x_0\in\bar{\Omega}$. Similar to the proof of Theorem~\ref{th4}, we have the following three cases.

{\it Case 1: $x_0\in\Omega\setminus \{0\}$.} Let $\nu_k=M_k^{-(p-1)/2}$ and 
\begin{align}\label{z20a}
&w_k(y)=M_k^{-1}u_k(x_k+\nu_k y, t_k),\\
&\tilde{w}_k(y)=M_k^{-p}\partial_t u_k(x_k+\nu_k y, t_k).\label{z20b}
\end{align}
Then we have
\begin{align}\label{z17}
\begin{cases}
\Delta w_k+|x_k+\nu_k y|^aw_k^p=\tilde{w}_k&\text{ in } \Omega_k\\
w_k=0 &\text{ on } \partial \Omega_k
\end{cases}
\end{align}
where $\Omega_k=\nu_k^{-1}(\Omega-x_k)$. Moreover, $0\leq w_k\leq 1=w_k(0)$. Now passing to the limit we obtain a contradiction with elliptic Liouville-type theorem for Lane-Emden equation \cite{GS81b}. We only have to show that the functions $w_k$ are locallly uniformly H\"older continuous and $\tilde{w}_k\to 0$ in $L^r_{loc}({\mathbb R}^N)$ with some $r>N/2$.

Let $R>0$. Using (\ref{z16}) we obtain for $k$ large enough,
\begin{align}\notag
\left(\int_{B_R} | \tilde{w}_k(y)|^rdy\right)^{1/r} &=M_k^{-p}\left(\int_{B_R} | \partial_t u_k(x_k+\nu_k y, t_k)|^rdy\right)^{1/r}\\
&=M_k^{-p}\nu_k^{-N/r}\left(\int_{\Omega} | \partial_t u_k(x, t_k)|^rdx\right)^{1/r}\notag\\
&\leq M_k^{-p}M_k^{N(p-1)/(2r)}M_k^{\frac{b}{2}+\beta(p-1)}=CM_k^{\gamma_1}.\label{az2}
\end{align}
where $\gamma_1=-p+\frac{N(p-1)}{4}+\frac{\alpha+1}{2\alpha}(p+1-q)$ and $r\in[2,\infty)$.

By taking $q$ close to $(p+1)/2$ and $\alpha$ sufficiently large, $\gamma_1$ will be negative provided $(N-3)p<N-1$, which is always true since $N\leq 4$ and $p\leq \frac{N+2+a}{N-2+a}$.

{\it Case 2: $x_0\in\partial \Omega$.} By the same argument as in the Case 1, we have the contradiction with Liouville-type theorem for Lane-Emden equation if $d(x_k,\partial \Omega)\nu_k^{-1}\to \infty$, or with  Liouville-type theorem for Lane-Emden equation in half-space if $d(x_k,\partial \Omega)\nu_k^{-1}$ is bounded.

{\it Case 3: $x_0=0$.} We have the following two possibilities:

(i) If $M_k^{(p-1)/(2+a)}|x_k|\leq C$, let $\nu_k=M_k^{-(p-1)/(2+a)}$ then $\nu_k^{-1}x_k$ is bounded. We may assume that $\nu_k^{-1}x_k\to P$. Let $w_k, \tilde{w}_k$ defined as in (\ref{z20a}) and (\ref{z20b}), by the same procedures, it is sufficient to show that functions $w_k$ are locallly uniformly H\"older continuous and $\tilde{w}_k\to 0$ in $L^r_{loc}({\mathbb R}^N)$ with some $r>N/2$.

Similarly as in (\ref{az2}), 
\begin{align*}
\left(\int_{B_R} | \tilde{w}_k(y)|^rdy\right)^{1/r} \leq M_k^{-p}M_k^{N(p-1)/(2r+ ar)}M_k^{\frac{b}{2}+\beta(p-1)}\leq CM_k^{\gamma_2}.
\end{align*}
where $r\in [2,\infty)$ and  $\gamma_2=-p+\frac{N(p-1)}{4}+\frac{\alpha+1}{2\alpha}(p+1-q)+\frac{N(p-1)}{r}(\frac{1}{2+a}-\frac{1}{2})\leq \gamma_1$. By taking $q$ close to $(p+1)/2$ and $\alpha$ sufficiently large, $\gamma_2$ will be negative provided $(N-3)p<N-1$.

(ii) If there exists a subsequence of $k$, still denoted by $k$, such that $M_k^{(p-1)/(2+a)}|x_k|\to \infty$. We can choose $m_k>1$ such that
$$M_k^{(p-1)/(2m_k+a)}|x_k|=1.$$
Let $\nu_k= M_k^{-m_k(p-1)/(2m_k+a)}$,  we may assume that $\nu_k^{-1/m_k}x_k\to P$ with $|P|=1$,  and $\nu_k^{1-1/m_k}\to l\in [0,1]$. Let $w_k, \tilde{w}_k$ defined as in (\ref{z20a}) and (\ref{z20b}), by the same procedures, it is sufficient to show that functions $w_k$ are locallly uniformly H\"older continuous and $\tilde{w}_k\to 0$ in $L^r_{loc}({\mathbb R}^N)$ with some $r>N/2$.

Similarly as in (\ref{az2}),
\begin{align*}
\left(\int_{B_R} | \tilde{w}_k(y)|^rdy\right)^{1/r} \leq M_k^{-p}M_k^{\frac{N}{r}\frac{m_k(p-1)}{2m_k+a})}M_k^{\frac{b}{2}+\beta(p-1)}\leq CM_k^{\gamma_3}.
\end{align*}
where $r\in [2,\infty)$ and  $\gamma_3=-p+\frac{N(p-1)}{4}+\frac{\alpha+1}{2\alpha}(p+1-q)+\frac{N(p-1)}{r}(\frac{m_k}{2m_k+a}-\frac{1}{2})\leq \gamma_1$. By taking $q$ close to $(p+1)/2$ and $\alpha$ sufficiently large, $\gamma_3$ will be negative provided $(N-3)p<N-1$.
Theorem is proved. 
\qed

{\bf Acknowledgement.} The author would like to thank Professor Philippe Souplet for valuable suggestions and comments.

\begin{appendix}
\section{}
\begin{lemma}\label{lemappe1}
 Assume that $0\in \Omega$, $a>-2$, $N\geq 2$ and $u$ is  solution of (\ref{1}) in $\Omega\times(0,T)$ in the sense of (\ref{solclass}). Then 
 $u$ is distributional solution in the sense 
\begin{align}
 -\int_{0}^{T}\int_{\Omega}(u(\varphi_t+\Delta\varphi)dxdt=\int_{0}^{T}\int_{\Omega}|x|^au^p\varphi dxdt 
\end{align}
for all $\varphi\in C^{\infty}_{0}(\Omega\times (0,T))$.
 \end{lemma}
\begin{proof} We follow the argument in \cite{PhS}. 
 If $a\geq 0$ then the result is immediate, so we may assume  $-2<a<0$.

Denote $d\sigma_\rho$ the surface measure on the sphere $\{x\in {\mathbb R}^N: |x|=\rho\}$. For $0<\eps<R$ such that $B_R\subset\subset\Omega$, for any $\tau>0$, we have
\begin{align}\label{2c1}
\int\limits_{\tau}^{T-\tau}\int\limits_{B_R\setminus B_{\eps}}|\nabla u|^2dxdt=&-\int\limits_{\tau}^{T-\tau}\int\limits_{B_R\setminus B_{\eps}}u\Delta u dxdt 
+\int\limits_{\tau}^{T-\tau}\int\limits_{|x|=R}uu' d\sigma_R dt-\int\limits_{\tau}^{T-\tau}\int\limits_{|x|=\eps}uu'd\sigma_\eps dt\notag\\
=&\int\limits_{\tau}^{T-\tau}\int\limits_{B_R\setminus B_{\eps}}|x|^au^{p+1}dxdt +\int\limits_{B_R\setminus B_{\eps}}(u(\tau)-u(T-\tau))dxdt \notag\\
&+ \int\limits_{\tau}^{T-\tau}\int\limits_{|x|=R}uu'd\sigma_Rdt-\int\limits_{\tau}^{T-\tau}\int\limits_{|x|=\eps}uu'd\sigma_\eps dt.
\end{align}
On the other hand, we have 
$$\int\limits_{\tau}^{T-\tau}\int\limits_{|x|=\eps}uu'\,d\sigma_\eps=\eps^{N-1}f'(\eps),
\quad\hbox{ where } f(r):=\frac{1}{2}\int\limits_{\tau}^{T-\tau}\int\limits_{S^{N-1}}u^2(r,\theta)\,d\theta.$$
Since $f\in C^1((0,R])\cap C([0,R])$ due to our regularity assumption (\ref{solclass}), we infer the existence of a sequence $\eps_i\to 0$
such that $\lim_{i\to\infty}\eps_i f'(\eps_i)=0$. Passing to the limit in (\ref{2c1}) with $\eps=\eps_i$ and noting $a>-N$, we have 
$$\int\limits_{\tau}^{T-\tau}\int\limits_{B_R}|\nabla u|^2\,dx <\infty.$$
Hence, there exist $\rho_i\to 0^{+}$ (depending on $\tau$)  such that 
\begin{align*}
 \int\limits_{\tau}^{T-\tau}\int\limits_{|x|=\rho_i} \rho_i|\nabla u|^2 \,d\sigma_{\rho_i}dt \to 0.
\end{align*}
Consequently, 
\begin{align}\label{app2}
\int\limits_{\tau}^{T-\tau}\int\limits_{|x|=\rho_i} |\nabla u| \,d\sigma_{\rho_i}\leq C\left((T-2\tau)\rho_i^{N-1}\int\limits_{|x|=\rho_i} |\nabla u|^2 \,d\sigma_{\rho_i}\right)^{1/2} \to 0.
\end{align}
Let now $\varphi\in C^\infty_0(\Omega\times (0,T))$ and denote $\Omega_\eps=\Omega\cap\{|x|>\rho\}$ for $\rho>0$ small.
From (\ref{1}), using Green's formula, we obtain
\begin{align}\label{z1}
\Bigl|\int\limits_{0}^{T}\int\limits_{\Omega_\rho} |x|^au^p\varphi\,dx+\int\limits_{0}^{T}\int\limits_{\Omega_\rho} u(\varphi_t+\Delta\varphi)\,dx\Bigr|
&=\Bigl|-\int\limits_{0}^{T}\int\limits_{\Omega_\rho} \varphi\Delta u\,dx+\int\limits_{0}^{T}\int\limits_{\Omega_\rho} u\Delta\varphi\,dx\Bigr|\notag\\
&=\Bigl|\int\limits_{0}^{T}\int\limits_{|x|=\rho} \varphi\,\frac{\partial u}{\partial r} \,d\sigma_\rho-\int\limits_{0}^{T}\int\limits_{|x|=\eps}u\frac{\partial \varphi}{\partial r} \,d\sigma_\rho\Bigr|.
\end{align}
Passing (\ref{z1}) to the limit with $\rho=\rho_i$, we conclude that $u$ is a distributional solution of (\ref{1}).
\end{proof}

\begin{lemma}\label{lemappe2}
 Assume that  $u$ is bounded solution of (\ref{1}) in ${\mathbb R }^N\times[0,T)$ in the sense of (\ref{solclass}). Assume in addition that $u$ is distributional solution. Then 
$u$ is integral solution in the sense that
\begin{align}
 u(t)=e^{t\Delta}u(0)+\int_{0}^{t}e^{(t-s)\Delta}(|.|u^p(s))ds
\end{align}
 \end{lemma}

\begin{proof} Lemma is standard for $a\geq 0$, so we only need to treat the case $a<0$. For any $\varepsilon>0$, let us consider following problem
\begin{align*}
 \begin{cases}
  \partial_t v_{\varepsilon}-\Delta v_{\varepsilon}=(|x|+\varepsilon)^au^p,\\
v_{\varepsilon}(0)=u(0).
 \end{cases}
\end{align*}
Then $v_{\varepsilon}\leq u$ by comparison property. Since $v_{\varepsilon}$ is increasing as $\varepsilon\to 0^{+}$, setting that $v_{\varepsilon} \to v$,  by motonone convergence,  we have
$$v(t)=e^{t\Delta}u(0)+\int_{0}^{t}e^{(t-s)\Delta}(|.|^au^p(s))ds.$$
It suffices to show that $u=v$. 

Let $z=u-v$. Then $z$ is a bounded, nonnegative distributional solution of $z_t-\Delta z=0$ in $Q_T:=R^N\times (0,T)$. By parabolic regularity (see e.g. \cite[Remark 48.3]{QS07}, we deduce that $z\in C^{2,1}(Q_T)$. Moreover, since $u,v\in C(\bar Q_T)$, it follows that $z\in C(\bar Q_T)$ with $z\equiv 0$ at $t=0$.
By standard uniqueness properties (see e.g. \cite[Theorem 2.4]{Wat89}, we conclude that $z=0$ in $Q_T$.

\end{proof}

In dimension $N = 1$, we have assumed $a>-1$ in order to make sense of distributional solutions. Actually, the definition (\ref{solclass}) is no longer consistent for $a < 0$ and $N = 1$ since $\Omega\setminus{0}$ is no longer connected and the problem should require boundary conditions at $x=0$. The following result shows that, for $N = 1$ and $a \in (-1, 0)$, there even exist solutions in the sense (\ref{solclass}) which are not distributional solutions.
\begin{proposition}\label{reappe}
 Let $N=1$ and $a\in (-1,0)$, then there exists solution $u$ of (\ref{1}) in $(-1,1)\times (0,1)$, but $u$ is not distributional solution.
\end{proposition}
\begin{proof}
 Let $B$ be unit ball in ${\mathbb R}^3$ and $v=v(r)$ be an regular positive radial solution in $B$ of the following Hardy elliptic equation
\begin{align} \label{ellip}
-\Delta v=|x|^{p-1+a}v^p,\quad  v(0)>0, \;v(1)=0.
 \end{align}
Since $p<p_S(p-1+a)=5+2(p-1+a)$, the existence of such function $v$ with homogeneous Dirichlet boundary condition was shown in \cite[Theorem 1.6 (iii)]{BVG10}. 

Let $w(r)=rv(r)$ then $w''=r^aw^p, \; r\in (0,1)$ and $w(0)=0$, $w'(0)=v(0)>0$. We set 
$$u(x,t)=w(|x|), \quad  (x,t)\in (-1,1)\times(0,1).$$
Then $u$ is solution of (\ref{1}) in the sense of (\ref{solclass})  in $(-1, 1)\times (0,1)$ with $u_t\equiv 0$. On the other hand, $u_x(0^{+},)=v(0)$, $ u_x(0^{-},t)=-v(0)$. This implies $u_{xx}(0,t)$ has a Dirac $2v(0)\delta_0$. Therefore, $u$ is no longer distributional solution.
\end{proof}

\end{appendix}



\end{document}